\documentclass[12pt]{article} 
\usepackage{amsfonts,amsmath,latexsym,amssymb,mathrsfs,amsthm,comment}
\usepackage{slashbox}
\usepackage{caption}

\let\OLDthebibliography\thebibliography
\renewcommand\thebibliography[1]{
  \OLDthebibliography{#1}
  \setlength{\parskip}{0pt}
  \setlength{\itemsep}{0pt plus 0.0ex}
}

\def\numberlikeadb{\global\def\theequation{\thesection.\arabic{equation}}}
\numberlikeadb
\newtheorem{theorem}{Theorem}[section]
\newtheorem{lemma}[theorem]{Lemma}
\newtheorem{corollary}[theorem]{Corollary}

\newtheorem{remark}[theorem]{Remark}

\usepackage{color}

\usepackage{lscape}
\usepackage{caption}
\usepackage{multirow}
\begin{document}

\title{New error bounds for Laplace approximation via Stein's method}
\author{Robert E. Gaunt\footnote{Department of Mathematics, The University of Manchester, Oxford Road, Manchester M13 9PL, UK}}

\date{} 
\maketitle

\vspace{-5mm}

\begin{abstract}We use Stein's method to obtain explicit bounds on the rate of convergence for the Laplace approximation of two different sums of independent random variables; one being a random sum of mean zero random variables and the other being a deterministic sum of mean zero random variables in which the normalisation sequence is random.  We make technical advances to the framework of Pike and Ren \cite{pike} for Stein's method for Laplace approximation, which allows us to give bounds in the Kolmogorov and Wasserstein metrics.  Under the additional assumption of vanishing third moments, we obtain faster convergence rates in smooth test function metrics.  As part of the derivation of our bounds for the Laplace approximation for the deterministic sum, we obtain new bounds for the solution, and its first two derivatives, of the Rayleigh Stein equation.
\end{abstract}


\noindent{{{\bf{AMS 2010 Subject Classification:}}} Primary 60F05; 62E17

\section{Introduction}

The central limit theorem states that for a sequence of independent and identically distribution (i.i.d.) random variables, $X_1,X_2,\ldots$, with zero mean and variance $\sigma^2\in(0,\infty)$, the standardised sum $W_n=\frac{1}{\sigma\sqrt{n}}\sum_{i=1}^n X_i$ convergences in distribution to the standard normal distribution, as $n\rightarrow\infty$.  By modifying the sum $W_n$ appropriately such that either the number of terms in the sum is random or the normalisation is random we can instead naturally arrive at an asymptotic Laplace distribution.  
Studying the rate of convergence to the Laplace distribution in these two settings, via Stein's method, is the subject of this paper. 

More precisely, consider the Laplace distribution with parameters $a\in\mathbb{R}$ and $b\in(0,\infty)$ with probability density function 
\begin{equation}\label{pdfl}f_W(x)=\frac{1}{2b}\mathrm{e}^{-\frac{|x-a|}{b}}, \quad x\in\mathbb{R}.
\end{equation}
If a random variable $W$ has density (\ref{pdfl}), then we write $W\sim\mathrm{Laplace}(a,b)$. It is readily checked that $\mathbb{E}[W]=a$ and $\mathrm{Var}(W)=2b^2$. For a comprehensive account of the properties and applications of the Laplace distribution, see \cite{kkp01}.  

The first limit theorem we consider concerns geometric sums, which arise in a variety of settings \cite{k97}. Let  $X_1,X_2,\ldots$ be a sequence of i.i.d$.$ random variables with zero mean and variance $\sigma^2\in(0,\infty)$ and let $N_p\sim\mathrm{Geo}(p)$ be independent of the $X_i$ with probability mass function $P(N_p=k)=p(1-p)^{k-1}$, $k=1,2,\ldots$, $0<p<1$.  Then, with an obvious abuse of notation, 
\[S_p:=\sqrt{p}\sum_{i=1}^{N_p}X_i\rightarrow_d\mathrm{Laplace}(0,\frac{\sigma}{\sqrt{2}}),\quad p\rightarrow0.\]
  This result is proved under the stronger assumption of symmetric $X_i$ in \cite{kkp01}, whilst weaker Lindeberg-type conditions for the existence of the distributional limit are given by \cite{toda}.  

The second limit theorem considered in this paper concerns the case in which the sum $\sum_{i=1}^nX_i$ is normalised by a random variable.  Let $B_n$ be a beta random variable with parameters $1$ and $n\geq1$ and probability density function
\[f_{B_n}(x)=n(1-x)^{n-1}, \quad 0<x<1.\]
We write $B_n\sim\mathrm{Beta}(1,n)$. As in the first limit theorem, let $X_1,X_2,\ldots$ be a sequence of i.i.d$.$ random variables with zero mean and variance $\sigma^2\in(0,\infty)$.  For $n\geq2$, let $B_{n-1}\sim \mathrm{Beta}(1,n-1)$ be independent of the $X_i$.  Then, Proposition 2.2.12 of \cite{kkp01} states that
\[T_n:=B_{n-1}^{1/2}\sum_{i=1}^n X_i\rightarrow_d\mathrm{Laplace}(0,\frac{\sigma}{\sqrt{2}}),\quad n\rightarrow\infty.\]
For characterisations of the Laplace distribution involving the random variables $S_p$ and $T_n$, see \cite{k84} and \cite{pakes1, pakes2}, respectively. 

In this paper, we give explicit bounds on the distance, with respect to certain probability metrics, between the distributions of $S_p$ and $T_n$ and their limiting Laplace distributions via Stein's method, a powerful probabilistic technique that was introduced in 1972 by Charles Stein \cite{stein} for normal approximation.  For a given target
distribution $q$, the first step in Stein's method is to find a
suitable operator $\mathcal{A}$ acting on a class of functions $\mathcal{F}$
such that $\mathbb{E} [\mathcal{A}f(Y)] =0$ for all $f\in\mathcal{F}$ if and only
if the random variable $Y$ has distribution $q$.  For the $N(\mu,\sigma^2)$ distribution, the
classical \emph{Stein operator} is $\mathcal{A}f(x)=\sigma^2f'(x)-(x-\mu)f(x)$.  This leads
to the \emph{Stein equation}
\begin{equation}\label{steineqn}\mathcal{A}f_h(x)=h(x)-\mathbb{E} [h(Y)],
\end{equation}
where the test function $h$ is real-valued.  The second step is to solve (\ref{steineqn}) for $f_h$ (for which we require $f_h\in\mathcal{F}$) and obtain suitable bounds for the solution. Finally, to approximate the distribution of a random variable of interest $W$ by the target distribution $q$, one
may evaluate both sides of (\ref{steineqn}) at $W$, take expectations, absolute values, and suprema of both sides over a class of functions $\mathcal{H}$ to obtain 
\begin{equation*}d_{\mathcal{H}}(W,Y):=\sup_{h\in\mathcal{H}}|\mathbb{E}[h(W)]-\mathbb{E}[h(Y)]|=\sup_{h\in\mathcal{H}}|\mathbb{E}[\mathcal{A}f_h(W)]|.
\end{equation*}
This is of interest because many important probability metrics are of the form $d_{\mathcal{H}}(W,Y)$, and in many settings bounding the expectation $\mathbb{E}[Af_h(W)]$ is relatively tractable.  In particular, taking
\begin{align*}\mathcal{H}_{\mathrm{K}}&=\{\mathbf{1}(\cdot\leq z)\,|\,z\in\mathbb{R}\}, \\
\mathcal{H}_{\mathrm{W}}&=\{h:\mathbb{R}\rightarrow\mathbb{R}\,|\,\text{$h$ is Lipschitz, $\|h'\|\leq1$}\}, \\
\mathcal{H}_{\mathrm{BW}}&=\{h:\mathbb{R}\rightarrow\mathbb{R}\,|\,\text{$h$ is Lipschitz, $\|h\|\leq1$ and $\|h'\|\leq1$}\}, \\
\mathcal{H}_{2}&=\{h:\mathbb{R}\rightarrow\mathbb{R}\,|\,\text{$h'$ is Lipschitz, $\|h''\|\leq1$}\}, \\
\mathcal{H}_{1,2}&=\{h:\mathbb{R}\rightarrow\mathbb{R}\,|\,\text{$h'$ is Lipschitz, $\|h'\|\leq1$ and $\|h''\|\leq1$}\}
\end{align*}
gives the Kolmogorov, Wasserstein and bounded Wasserstein distances, which we denote by $d_{\mathrm{K}}$, $d_{\mathrm{W}}$ and $d_{\mathrm{BW}}$, respectively, as well as two smooth test function metrics, which we denote by $d_2$ and $d_{1,2}$, respectively. (Here and throughout the paper $\|g\|:=\|g\|_\infty=\sup_{x\in\mathbb{R}}|g(x)|$.)  The $d_2$ and $d_{1,2}$ and similar smooth test function metrics are often found in applications of Stein's method in which `fast' convergence rates are sought, see, for example, \cite{bj2,f18,gaunt chi square,goldstein4}.  


Stein's method was adapted to the Laplace distribution by \cite{pike} (a number of their contributions are outlined in Section \ref{sec2}), and as an application they derived an explicit bound on the bounded Wasserstein distance between the distribution of $S_p$ and its limiting Laplace distribution.  Their approach, which involves the introduction of the so-called \emph{centered equilibrium transformation} for Laplace approximation, mirrored that of \cite{pekoz1}, who used Stein's method for exponential approximation to give explicit bounds on the rate of convergence in a generalisation of a well-known result of R\'{e}nyi \cite{renyi} concerning the convergence of geometric sums of positive random variables to the exponential distribution.  In this paper, we make technical improvements on the work of \cite{pike} (through Lemma \ref{lemivp} and Theorem \ref{jazzz}) that allow for their framework of Laplace approximation by Stein's method to yield optimal order Kolmogorov and Wasserstein distance bounds, as well as faster convergence rates in the $d_2$ distance.  As an application we are able to obtain the following theorem.


\begin{theorem}\label{thm555}Suppose $X_1,X_2,\ldots$ is a sequence of independent random variables with $\mathbb{E}[X_i]=0$ and $\mathbb{E}[X_i^2]=\sigma^2\in(0,\infty)$. Let $N_p\sim\mathrm{Geo}(p)$, $0<p<1$, be independent of the $X_i$.  Define $S_p=\sqrt{p}\sum_{i=1}^{N_p} X_i$ and let $Z\sim \mathrm{Laplace}(0,\frac{\sigma}{\sqrt{2}})$.  Then
\begin{equation}\label{wedfg}d_{\mathrm{K}}(S_p, Z)\leq \sqrt{2}\bigg(\frac{7}{2}+\sqrt{10}\bigg)\frac{\sqrt{p}}{\sigma}\sup_{i\geq1}\|F_{X_i}^{-1}-F_{X_i^L}^{-1}\|.
\end{equation}
Suppose additionally that $\rho_3=\sup_{i\geq1}\mathbb{E}[|X_i|^3]<\infty$. Then
\begin{equation}\label{rwrwa}d_{\mathrm{W}}(S_p,Z)\leq 2\sigma\sqrt{p}\bigg(1+\frac{\rho_3}{3\sigma^3}\bigg).
\end{equation}
Let $k\geq 1$. Suppose that $\rho_{k+2}=\sup_{i\geq1}\mathbb{E}[|X_i|^{k+2}]<\infty$.  Then
\begin{equation}\label{taubound}d_{\mathrm{K}}(S_p, Z)\leq 11.56\cdot2^{\frac{k-1}{k+1}}(2p)^{\frac{k}{2(k+1)}}\bigg(\frac{\rho_k}{\sigma^k}+\frac{2\rho_{k+2}}{(k+1)(k+2)\sigma^{k+2}}\bigg)^{\frac{1}{(k+1)}}.
\end{equation}
Finally, suppose that $X_1,X_2,\ldots$ are identically distributed and that $\mathbb{E}[X_1^3]=0$, $\mathbb{E}[X_1^4]<\infty$.
 Then
\begin{align}\label{on11}d_2(S_p,Z)\leq \sigma^2p\bigg[\frac{2-p}{1-p}+\frac{\mathbb{E}[X_1^4]}{6\sigma^4}+\frac{\sqrt{p}\log(1/p)}{\sqrt{2}(1-p)}\bigg(2+\frac{\mathbb{E}[|X_1|^3]}{\sigma^3}\bigg)\bigg].
\end{align}
\end{theorem}

\begin{remark}\label{rety}
The dependence on $p$ in (\ref{taubound}) is worse than in (\ref{wedfg}), but the bound may be preferable if $\sup_{i\geq1}\|F_{X_i}^{-1}-F_{X_i^L}^{-1}\|$ is difficult to compute or large.  Note, though, that as $k$ increases the exponent $\frac{k}{2(k+1)}$ of $p$ in (\ref{taubound}) approaches the exponent $\frac{1}{2}$ of (\ref{wedfg}).

\end{remark}

We are also able to obtain a similar theorem for the deterministic sum $T_n$:

\begin{theorem}\label{thm888}Let $n\geq2$ and suppose that $X_1,\ldots,X_n$ are independent random variables with $\mathbb{E}[X_i]=0$, $\mathbb{E}[X_i^2]=\sigma^2\in(0,\infty)$ and $\mathbb{E}[|X_i|^3]<\infty$, for all $1\leq i\leq n$.  Then
 \begin{equation*}d_{\mathrm{K}}(T_n,Z)\leq\frac{0.5600}{\sigma^3n^{3/2}}\sum_{i=1}^n\mathbb{E}[|X_i|^3]+\frac{1}{n}\bigg(1+2\bigg(1-\frac{2}{n}\bigg)^{n-2}\bigg).
\end{equation*}
and
\begin{equation*}d_{\mathrm{W}}(T_n,Z)\leq\frac{2\sqrt{2}\sigma}{3n^{3/2}}\sum_{i=1}^n\bigg(2+\frac{\mathbb{E}[|X_i|^3]}{\sigma^3}\bigg)+\frac{9.168\sigma}{n}.
\end{equation*}
In addition to the above assumptions, suppose that $\mathbb{E}[X_i^3]=0$ and $\mathbb{E}[X_i^4]<\infty$, for all $1\leq i\leq n$. Then
\begin{equation*}d_{1,2}(T_n,Z)\leq\frac{\sigma^2}{n^2}\sum_{i=1}^n\bigg(1+\frac{\mathbb{E}[X_i^4]}{3\sigma^4}\bigg)+\frac{9.168\sigma}{n}.
\end{equation*}
\end{theorem}

Written in the notation of Theorem \ref{thm555}, the bounded Wasserstein distance bound of \cite{pike} reads $d_{\mathrm{BW}}(S_p,Z)\leq\sigma\sqrt{p}\big(1+\frac{2\sqrt{2}}{\sigma}\big)\big(1+\frac{\rho_3}{3\sigma^3}\big)$. We see that in addition to being given in a stronger metric, the Wasserstein distance bound (\ref{rwrwa}) of Theorem \ref{thm555} has a better dependence on $\sigma$ (the bound of \cite{pike} has an extra factor of $\big(1+\frac{2\sqrt{2}}{\sigma}\big)$ meaning that the bound has a worse dependence on $\sigma$ if $\sigma$ is `small') and a smaller numerical constant if $\sigma<2\sqrt{2}$ (the bound of \cite{pike} has the smaller numerical constant if $\sigma>2\sqrt{2}$).
The bound (\ref{rwrwa}) also improves on the recent Wasserstein distance bound given in Theorem 5.10 of \cite{gaunt vgii}, in which Laplace approximations were obtained as part of a more general work on variance-gamma approximation.  By working in a specialist Laplace framework, it is no surprise that we outperform the results of \cite{gaunt vgii}, and our Kolmogorov distance bound (\ref{wedfg}) is also an improvement on the analogous bound in Theorem 5.10 of that work.  The $O(p)$ bound (\ref{on11}) is the first faster than $O(p^{1/2})$ bound for the random sum $S_p$ in the literature.  The faster convergence rate is a result of the vanishing third moment assumption, and as such complements a number of other `matching moments' limit theorems that are found in the Stein's method literature, see, for example, \cite{daly, f18, gaunt rate, gaunt fon, goldstein, lefevre}.  Theorem \ref{thm888} gives the first bounds in the literature on the rate of convergence of the deterministic sum $T_n$ to its asymptotic Laplace distribution.  Again, under the assumption of vanishing third moments, we obtain a faster convergence rate.  As part of our proof of the theorem, we obtain the first bounds in the literature for the solution, and its first two derivatives, of the Rayleigh Stein equation, which may be useful in future applications.

The rest of the paper is organised as follows.  In Section \ref{sec2}, we obtain new bounds for the solution of the Laplace Stein equation (Lemma \ref{lemivp}) and give general bounds for Laplace approximation involving the centered equilibrium distribution (Theorem \ref{jazzz}).  In Sections \ref{sec3} and \ref{sec4}, we prove Theorems \ref{thm555} and \ref{thm888}, respectively.  In Section \ref{sec5}, we obtain new bounds for the solution of the Rayleigh Stein equation that are used in the proof of Theorem \ref{thm888}.

\section{Stein's method for the Laplace distribution}\label{sec2}


In this section, we recall some of the theory developed by \cite{pike} for Stein's method for Laplace approximation and make some technical improvements that allow their framework for Laplace approximation to be applied in the Kolmogorov and Wasserstein metrics, as well as the $d_2$ metric when faster convergence rates are sought.  We begin by recalling the following characterisation of the Laplace distribution \cite[Theorem 1.1]{pike}.

Let $W$ be a real-valued random variable. Then $W$ follows the $\mathrm{Laplace}(0,b)$ distribution if and only if
\begin{equation}\label{alapchar} \mathbb{E}\big[b^2f''(W)-f(W)+f(0)\big]=0
\end{equation}
for all $f:\mathbb{R}\rightarrow\mathbb{R}$ such that $f$ and $f'$ are locally absolutely continuous and $\mathbb{E}|f'(Z)|<\infty$ and $\mathbb{E}|f''(Z)|<\infty$, for $Z\sim\mathrm{Laplace}(0,b)$.  Based on this characterisation, \cite{pike} were led to the initial value problem
\begin{equation}\label{ivp}b^2f''(x)-f(x)=\tilde{h}(x), \quad f(0)=0,
\end{equation}
where $\tilde{h}(x)=h(x)-\mathbb{E}[h(Z)]$, $Z\sim\mathrm{Laplace}(0,b)$.

At this point it is worth noting that an alternative Stein equation for the $\mathrm{Laplace}(0,b)$ distribution is given by
\begin{equation}\label{ivp22}b^2xf''(x)+2b^2f'(x)-xf(x)=\tilde{h}(x),
\end{equation}
which is a special case of the variance-gamma Stein equation of \cite{gaunt vg} (it is noted in Proposition 1.2 of \cite{gaunt vg} that the Laplace distribution is a special case of the variance-gamma distribution).  A framework for variance-gamma approximation by Stein's method in the Kolmogorov and Wasserstein metrics was developed by \cite{gaunt vgii}, and a special case of this general framework gives a framework for Laplace approximation.  However, the Stein equation (\ref{ivp22}) is more difficult to work with than (\ref{ivp}) and it is therefore not surprising that all the comparable results for Laplace approximation obtained in this paper outperform those of \cite{gaunt vgii}.  We also remark that another Stein characterisation of the Laplace distribution is given by \cite{ah19}, as a special case of a general characterisation concerning infinitely divisible distributions, although the quantitative limit theorems derived in their work are quite different to ours.

Let us now focus on the initial value problem (\ref{ivp}).  The solution
\begin{align}\label{ff9os}f(x)=\frac{1}{2b}\bigg(&\mathrm{e}^{x/b}\int_x^\infty \mathrm{e}^{-t/b}\tilde{h}(t)\,\mathrm{d}t+\mathrm{e}^{-x/b}\int_{-\infty}^x \mathrm{e}^{t/b}\tilde{h}(t)\,\mathrm{d}t\bigg)
\end{align}
was obtained by \cite{pike}, as well as bounds for $f$ and its first three derivatives.  In the following lemma, we improve on Lemma 2.2 of \cite{pike} by obtaining bounds for $f$ and its derivatives (of arbitrary order) that have smaller constants and hold for a larger class of functions.  The latter improvement is crucial in enabling us to later obtain Kolmogorov and Wasserstein distance bounds for Laplace approximation. 

\begin{lemma}\label{lemivp}Let $h:\mathbb{R}\rightarrow\mathbb{R}$ be a measurable function with $\mathbb{E}|h(Z)|<\infty$, where $Z\sim\mathrm{Laplace}(0,b)$.  Let $f$ be the solution (\ref{ff9os}) to the Stein equation (\ref{ivp}).
If $h$ is bounded, then this is the unique bounded solution to (\ref{ivp}).  Moreover, the solution $f$ and its first two derivatives satisfy the bounds
\begin{equation}\label{firstbounds}\|f\|\leq\|\tilde{h}\|, \quad \|f'\|\leq\frac{1}{b}\|\tilde{h}\|, \quad \|f''\|\leq\frac{2}{b^2}\|\tilde{h}\|.
\end{equation}
Suppose that $h$ is Lipschitz.  Then
\begin{equation*}|f(x)|\leq(2b+|x|)\|h'\|, \quad x\in\mathbb{R},
\end{equation*}
Now suppose that $h^{(k)}$ is Lipschitz, where $h^{(0)}\equiv h$.  Then, for $k\geq0$,
\begin{equation}\label{lipbounds}\|f^{(k+1)}\|\leq\|h^{(k+1)}\|, \quad \|f^{(k+2)}\|\leq\frac{1}{b}\|h^{(k+1)}\|, \quad \|f^{(k+3)}\|\leq\frac{2}{b^2}\|h^{(k+1)}\|.
\end{equation}
\end{lemma}

\begin{proof} It is easily verified that there is at most one bounded solution  to (\ref{ivp}).  Suppose that $u$ and $v$ are solutions to (\ref{ivp}).  Then $w=u-v$ satisfies $w(0)=0$ and solves the differential equation $b^2w''(x)-w(x)=0$, the general solution to which is given by  $w(x)=A\mathrm{e}^{x/b}+B\mathrm{e}^{-x/b}$.  For $w(x)$ to be bounded for all $x\in\mathbb{R}$, we must take $A=B=0$, from which we conclude that $w=0$, so that $u=v$.

Now we establish the bounds in (\ref{firstbounds}).  Suppose $h$ is bounded.   
 We first note that, for all $x\in\mathbb{R}$,
\begin{equation*}\bigg|\mathrm{e}^{x/b}\int_x^\infty \mathrm{e}^{-t/b}\tilde{h}(t)\,\mathrm{d}t\bigg|\leq \|\tilde{h}\|\mathrm{e}^{x/b}\int_x^\infty \mathrm{e}^{-t/b}\,\mathrm{d}t=b\|\tilde{h}\|,
\end{equation*}
and
\begin{equation*}\bigg|\mathrm{e}^{-x/b}\int_{-\infty}^x\mathrm{e}^{t/b}\tilde{h}(t)\,\mathrm{d}t\bigg|\leq \|\tilde{h}\|\mathrm{e}^{-x/b}\int_{-\infty}^x \mathrm{e}^{t/b}\,\mathrm{d}t=b\|\tilde{h}\|.
\end{equation*}
Applying these inequalities into (\ref{ff9os}) gives the bound
\begin{equation}\label{fequal}\|f\|\leq\frac{1}{2b}\big(b\|\tilde{h}\|+b\|\tilde{h}\|\big)=\|\tilde{h}\|.
\end{equation}
Differentiating both sides of (\ref{ff9os}) gives that
\begin{equation}\label{firstd}f'(x)=\frac{1}{2b}\bigg(\frac{1}{b}\mathrm{e}^{x/b}\int_x^\infty\mathrm{e}^{-t/b}\tilde{h}(t)\,\mathrm{d}t-\frac{1}{b}\mathrm{e}^{-x/b}\int_{-\infty}^x\mathrm{e}^{t/b}\tilde{h}(t)\,\mathrm{d}t\bigg),
\end{equation}
and so
\begin{equation*}\|f'\|\leq\frac{1}{2b}\big(\|\tilde{h}\|+\|\tilde{h}\|\big)=\frac{1}{b}\|\tilde{h}\|.
\end{equation*}
From (\ref{ivp}) and formula (\ref{ff9os}) we have that, for all $x\in\mathbb{R}$,
\begin{align*}|f''(x)|&=\frac{1}{b^2}|\tilde{h}(x)+f(x)| \\
&=\bigg|\frac{1}{b^2}\tilde{h}(x)+\frac{1}{2b^3}\bigg(\mathrm{e}^{x/b}\int_x^\infty\mathrm{e}^{-t/b}\tilde{h}(t)\,\mathrm{d}t+\mathrm{e}^{-x/b}\int_{-\infty}^x\mathrm{e}^{t/b}\tilde{h}(t)\,\mathrm{d}t\bigg)\bigg| \\
&\leq\frac{1}{b^2}\|\tilde{h}\|+\frac{1}{2b^3}\big(b\|\tilde{h}\|+b\|\tilde{h}\|\big)=\frac{2}{b^2}\|\tilde{h}\|.
\end{align*}

Now we suppose that $h$ is Lipschitz.  We shall now prove the non-uniform bound for $|f(x)|$.  By the mean value theorem, $|\tilde{h}(x)|\leq\|h'\|(|x|+\mathbb{E}|Z|)$, where $Z\sim \mathrm{Laplace}(0,b)$.  Note that $\mathbb{E}|Z|=b$.  Also, in anticipation of bounding $|f(x)|$ we note two integral inequalities: for $\lambda>0$ and $x\in\mathbb{R}$,
\begin{align*}\mathrm{e}^{\lambda x}\int_x^\infty |t|\mathrm{e}^{-\lambda t}\,\mathrm{d}t<\frac{2}{\lambda^2}(1+\lambda|x|), \quad \mathrm{e}^{-\lambda x}\int_{-\infty}^x |t|\mathrm{e}^{\lambda t}\,\mathrm{d}t<\frac{2}{\lambda^2}(1+\lambda|x|).
\end{align*} 
We verify the first inequality; the second inequality is proved similarly.  For $x\geq0$,
\begin{align*}\mathrm{e}^{\lambda x}\int_x^\infty |t|\mathrm{e}^{-\lambda t}\,\mathrm{d}t=\frac{1}{\lambda^2}(1+\lambda x)
\end{align*}
and, for $x<0$, 
\begin{align*}\mathrm{e}^{\lambda x}\int_x^\infty |t|\mathrm{e}^{-\lambda t}\,\mathrm{d}t&=\mathrm{e}^{\lambda x}\bigg(-\int_x^0 t\mathrm{e}^{-\lambda t}\,\mathrm{d}t+\int_0^\infty t\mathrm{e}^{-\lambda t}\,\mathrm{d}t\bigg) \\
&=\frac{1}{\lambda^2}\big(2\mathrm{e}^{\lambda x}-1-\lambda x\big)<\frac{1}{\lambda^2}(1-\lambda x).
\end{align*}
Putting all of the above together, we obtain, for $x\in\mathbb{R}$,
\begin{align*}|f(x)|&\leq\frac{\|h'\|}{2b}\bigg(\mathrm{e}^{x/b}\int_x^\infty\mathrm{e}^{-t/b}(|t|+b)\,\mathrm{d}t+\mathrm{e}^{-x/b}\int_{-\infty}^x\mathrm{e}^{t/b}(|t|+b)\,\mathrm{d}t\bigg) \\
&\leq\frac{1}{2b}\bigg(2b^2\bigg(1+\frac{|x|}{b}\bigg)+2b^2\bigg) =\|h'\|(2b+|x|).
\end{align*}

Finally, we prove the uniform bounds.  We note that applying integration by parts to (\ref{firstd}) gives that
\begin{align*}f'(x)&=\frac{1}{2b}\bigg\{\frac{1}{b}\mathrm{e}^{x/b}\bigg[b\mathrm{e}^{-x/b}\tilde{h}(x)+b\int_x^\infty\mathrm{e}^{-t/b}h'(t)\,\mathrm{d}t\bigg]\\
&\quad-\frac{1}{b}\mathrm{e}^{-x/b}\bigg[b\mathrm{e}^{x/b}\tilde{h}(x)-b\int_{-\infty}^x\mathrm{e}^{t/b}h'(t)\,\mathrm{d}t\bigg]\bigg\} \\
&=\frac{1}{2b}\bigg(\mathrm{e}^{x/b}\int_x^\infty\mathrm{e}^{-t/b}h'(t)\,\mathrm{d}t+\mathrm{e}^{-x/b}\int_{-\infty}^x\mathrm{e}^{t/b}h'(t)\,\mathrm{d}t\bigg).
\end{align*}
We recognise this representation of $f'(x)$ as being the same as the representation (\ref{ff9os}) of $f(x)$, with $\tilde{h}(t)$ replaced by $h'(t)$, and so we can immediately deduce the bounds in (\ref{lipbounds}) for $\|f'\|$, $\|f''\|$ and $\|f^{(3)}\|$.  Repeating the procedure inductively yields the bounds for $\|f^{(k+1)}\|$, $\|f^{(k+2)}\|$ and $\|f^{(k+3)}\|$, $k\geq0$.
\end{proof}


The following distributional transformation, introduced by \cite{pike}, is very natural in the context of Stein's method for Laplace approximation. Let $W$ have mean zero and non-zero finite variance. Then we say that the random variable $W^L$ has the \emph{centered equilibrium
distribution} with respect to $W$ if 
\begin{equation}\label{pikegh}\mathbb{E}[f(W)]-f(0)=\frac{1}{2}\mathbb{E}[W^2]\mathbb{E}[f''(W^L)]
\end{equation}
for all twice differentiable $f:\mathbb{R}\rightarrow\mathbb{R}$ such that $\mathbb{E}|f(W)|<\infty$ and $\mathbb{E}|Wf'(W)|<\infty$.  Stronger conditions were imposed on $f$ by \cite{pike}, but on examining the proof of their Theorem 3.2 it can be seen that the weaker conditions presented here are sufficient to ensure $W^L$ exists and is unique.
We also refer the reader to \cite{dobler} for a generalisation of (\ref{pikegh}) to all random variables $W$ with finite second moment, and we note that the centered equilibrium distribution is itself the Laplace analogue of the equilibrium distribution that is used in Stein's method for exponential approximation by \cite{pekoz1}.  Some useful properties of the centered equilibrium transformation are collected in Section 3 of \cite{pike} and Proposition 4.6 of \cite{gaunt vgii}.  In the sequel, the following moment relations will be important: assuming $\mathbb{E}[W^2]=2b^2$, we have that, for $r\geq0$,
\begin{equation}\label{moml}\mathbb{E}[(W^L)^r]=\frac{\mathbb{E}[W^{r+2}]}{(r+1)(r+2)b^2}, \quad \mathbb{E}[|W^L|^r]=\frac{\mathbb{E}[|W|^{r+2}]}{(r+1)(r+2)b^2}.
\end{equation}
The formulas in (\ref{moml}) are obtained by substituting $f_1(w)=w^{r+2}$ and $f_2(w)=|w|^{r+2}$, respectively, into (\ref{pikegh}) and using that $\mathbb{E}[W^2]=2b^2$.

 Theorem \ref{jazzz} below gives general bounds for Laplace approximation involving the centered equilibrium transformation. Bounds (\ref{dfgh1}) -- (\ref{zezozr2}) of the theorem are the Laplace analogues of the bounds of Theorem 2.1 of \cite{pekoz1}, which give Kolmogorov and Wasserstein distance bounds in terms the absolute difference between a random variable $W$ and its $W$-equilibrium transformation.  We additionally provide a bound in the weaker $d_2$ metric, which is used to obtain the $O(p^{-1})$ bound (\ref{on11}) of Theorem \ref{thm555}. We mostly follow the approach of \cite{pekoz1}, but the approach used to obtain the $d_2$ metric bound is similar to that used by \cite[Theorem 3.1]{goldstein} to prove an analogous bound for the zero bias transformation. We begin by stating three lemmas.  The proofs of Lemmas \ref{lemcon} and \ref{lemsmo} are simple and hence omitted, and the proof of Lemma \ref{hae} follows immediately from the estimates of Lemma \ref{lemivp}.

\begin{lemma}\label{lemcon}Let $Z\sim \mathrm{Laplace}(0,b)$.  Then, for any random variable $W$,
\begin{equation}\label{bbnn}\mathbb{P}(\alpha\leq W\leq \beta)\leq \frac{\beta-\alpha}{2b}+2d_{\mathrm{K}}(W,Z).
\end{equation} 
\end{lemma}

\begin{lemma}\label{hae}For any $a\in\mathbb{R}$ and any $\epsilon>0$, define
\begin{equation}\label{hae5}h_{a,\epsilon}(x):=\epsilon^{-1}\int_0^\epsilon \mathbf{1}(x+s\leq a)\,\mathrm{d}s.
\end{equation}
Let $f_{a,\epsilon}$ be the solution (\ref{ff9os}) with test function $h_{a,\epsilon}$.  Let $h_{a,0}(x)=\mathbf{1}(x\leq a)$ and define $f_{a,0}$ accordingly. Then
\begin{eqnarray}\label{hae1}\|f_{a,\epsilon}\|&\leq&1, \\
\label{hae2}\|f_{a,\epsilon}'\|&\leq& \frac{1}{b}, \\
\|f_{a,\epsilon}''\|&\leq& \frac{2}{b^2}.\nonumber
\end{eqnarray}
\end{lemma}

\begin{lemma}\label{lemsmo}Let $W$ be a real-valued random variable and let $Z\sim \mathrm{Laplace}(0,b)$.  Then, for any $\epsilon>0$,
\begin{equation*}d_{\mathrm{K}}(W,Z)\leq \frac{\epsilon}{2b}+\sup_{a\in\mathbb{R}}|\mathbb{E}[h_{a,\epsilon}(W)]-\mathbb{E}[h_{a,\epsilon}(Z)]|,
\end{equation*}
with $h_{a,\epsilon}$ defined as in Lemma \ref{hae}.
\end{lemma}

\begin{theorem} \label{jazzz} Let $W$ be random variable with zero mean and variance $2b^2\in(0,\infty)$, and let $W^{L}$ have the $W$-centered equilibrium distribution.  Then, for any $\beta>0$,  
\begin{eqnarray}\label{dfgh1}d_{\mathrm{K}}(W,Z)&\leq& \frac{(7/2+\sqrt{10})\beta}{b}+3\bigg(1+\sqrt{\frac{2}{5}}\bigg)\mathbb{P}(|W-W^{L}|>\beta), \\
\label{dk76}d_{\mathrm{K}}(W^{L},Z)&\leq& \frac{\beta}{b}+2\mathbb{P}(|W-W^{L}|>\beta).
\end{eqnarray}
Suppose further that $\mathbb{E}[|W|^3]<\infty$.  Then 
\begin{eqnarray}\label{zezozr}d_{\mathrm{W}}(W,Z) &\leq&2\mathbb{E}|W-W^{L}|, \\
\label{zezozr1}d_{\mathrm{W}}(W^{L},Z) &\leq&\mathbb{E}|W-W^{L}|, \\
\label{zezozr2} d_{\mathrm{K}}(W^{L},Z) &\leq&\frac{1}{b}\mathbb{E}|W-W^{L}|.
\end{eqnarray}
Suppose now that $\mathbb{E}[W^4]<\infty$.  Then
\begin{equation}\label{ordern}d_{2}(W,Z)\leq b\mathbb{E}[|\mathbb{E}[W-W^L\,|\,W]|] +\mathbb{E}[(W-W^L)^2].
\end{equation}
\end{theorem}

\begin{remark}Analogues of inequalities (\ref{dfgh1}) -- (\ref{zezozr2}) for variance-gamma approximation were given in Theorem 4.10 of \cite{gaunt vgii}, which as special cases give bounds for Laplace approximation in terms of the centered equilibrium distribution.  In all cases, our bounds improve on the bounds of \cite{gaunt vgii}.
\end{remark}

\begin{proof}For ease of notation, we let $\kappa=d_{\mathrm{K}}(W,Z)$. We also let $\Delta:= W -W^{L}$ and $I_1 := \mathbf{1}(|\Delta| \leq \beta)$.  Let $f$ be the solution of the $\mathrm{Laplace}(0,b)$ Stein equation with test function $h_{a,\epsilon}$, as given in (\ref{hae5}).  Note that the expectation $\mathbb{E}[ f''(W^{L})]$ is well defined, since $\|f''\|<\infty$ (see Lemma \ref{hae}).  By the Laplace Stein equation (\ref{ivp}), we have
\begin{align*}\mathbb{E}[h(W)]-\mathbb{E}[h(Z)]&=\mathbb{E}[b^2f''(W)-f(W)]\\
&=b^2\mathbb{E}[I_1(f''(W)-f''(W^{L}))]+b^2\mathbb{E}[(1-I_1)(f''(W)-f''(W^{L}))]\\
&=:J_1+J_2.
\end{align*}
Using the bound (\ref{hae1}) we have
\begin{align*}|J_2|&=|\mathbb{E}[(1-I_1)(f(W)-f(W^{L})+\tilde{h}_{a,\epsilon}(W)-\tilde{h}_{a,\epsilon}(W^{L}))]| \\
&=|\mathbb{E}[(1-I_1)(f(W)-f(W^{L})+h_{a,\epsilon}(W)-h_{a,\epsilon}(W^{L}))]| \\
&\leq (2\|f\|+1)\mathbb{P}(|\Delta|>\beta)\\
&\leq 3\mathbb{P}(|\Delta|>\beta).
\end{align*}
We also have
\begin{align*}J_1&=\mathbb{E}\bigg[I_1\int_0^{-\Delta} b^2f^{(3)}(W+t)\,\mathrm{d}t\bigg] \\
&=\mathbb{E}\bigg[I_1\int_0^{-\Delta} \big\{f'(W+t)-\epsilon^{-1}\mathbf{1}(a-\epsilon\leq W+t\leq a)\big\}\,\mathrm{d}t\bigg] \\
&\leq \|f'\|\mathbb{E}|I_1\Delta|+\epsilon^{-1}\int_{-\beta}^0\mathbb{P}(a-\epsilon\leq W+t\leq a)\,\mathrm{d}t \\
&\leq\frac{\beta}{b}+\frac{\beta}{2b}+2\beta\epsilon^{-1}\kappa = \frac{3\beta}{2b}+2\beta\epsilon^{-1}\kappa,
\end{align*}
where we used inequality (\ref{hae2}) and Lemma \ref{lemcon} to obtain the last inequality.  By a similar argument,
\begin{align*}J_1\geq  -\frac{3\beta}{2b}-2\beta\epsilon^{-1}\kappa,
\end{align*}
and so we conclude that
\begin{equation*}|J_1|\leq \frac{3\beta}{2b}+2\beta\epsilon^{-1}\kappa.
\end{equation*}
We now apply Lemma \ref{lemsmo} and take the convenient choice $\epsilon=\eta\beta$, $\eta>2$, to obtain
\begin{align*}\kappa&\leq 3\mathbb{P}(|\Delta|>\beta)+\frac{3\beta+\epsilon}{2b}+2\beta\epsilon^{-1}\kappa = 3\mathbb{P}(|\Delta|>\beta)+\frac{(3+\eta)\beta}{2b}+\frac{2\kappa}{\eta},
\end{align*}
which on rearranging yields
\begin{align}\label{psrev}\kappa\leq\frac{3\eta}{\eta-2}\mathbb{P}(|\Delta|>\beta)+\frac{(3\eta+\eta^2)\beta}{2b(\eta-2)}.
\end{align}
Choosing $\eta=2+\sqrt{10}$ minimises the second term in (\ref{psrev}) and yields the bound (\ref{dfgh1}). We elected to minimise the second term because in some applications the first term vanishes; as an example, see the proof of inequality (\ref{wdv}).

Now we prove inequality (\ref{dk76}).  We have
\begin{align*}\mathbb{E}[b^2f''(W^{L})-f(W^{L})]&=\mathbb{E}[f(W)-f(W^{L})] \\
&=\mathbb{E}[I_1(f(W)-f(W^{L}))]+\mathbb{E}[(1-I_1)(f(W)-f(W^{L}))].
\end{align*} 
By the mean value theorem, applying the triangle inequality and then using the bounds (\ref{hae1}) and (\ref{hae2}) we obtain
\begin{align*}\mathbb{E}[b^2f''(W^{L})-f(W^{L})] &\leq \|f'\|\mathbb{E}|I_1\Delta|+2\|f\|\mathbb{P}(|\Delta|>\beta)  \\
&\leq \frac{\beta}{b}+2\mathbb{P}(|\Delta|>\beta),
\end{align*}
yielding inequality (\ref{dk76}).

Now suppose that $\mathbb{E}[|W|^3]<\infty$. By the absolute moment relation (\ref{moml}), this assumption guarantees that $\mathbb{E}|W^{L}|<\infty$. Let $h\in\mathcal{H}_{\mathrm{W}}$. We have
\begin{align*}|\mathbb{E}[h(W)]-\mathbb{E}[h(Z)]|&=|\mathbb{E}[b^2f''(W)-f(W)]|=b^2|\mathbb{E}[f''(W)-f''(W^{L})]| \\
&\leq b^2\|f^{(3)}\|\mathbb{E}|W-W^{L}|\leq 2\mathbb{E}|W-W^{L}|,
\end{align*}
where we used the bound $\|f^{(3)}\|\leq\frac{2}{b^2}\|h'\|$ of Lemma \ref{lemivp} in the final step. This proves inequality (\ref{zezozr}).  Also,
\begin{align}\big|\mathbb{E}\big[b^2 f''(W^{L})-f(W^{L})\big]\big| &=\big|\mathbb{E}f(W)-\mathbb{E}f(W^{L})\big|\nonumber \\
\label{alige}&\leq  \|f'\|\mathbb{E}|W-W^{L}|.
\end{align}
Using inequality $\|f'\|\leq \|h'\|$ of Lemma \ref{lemivp} to (\ref{alige}) gives (\ref{zezozr1}). Suppose now that $h\in\mathcal{H}_{\mathrm{K}}$. Then using the bound $\|f'\|\leq\frac{1}{b}\|\tilde{h}\|$ gives us (\ref{zezozr2}). 

Finally, let $h\in \mathcal{H}_2$.  Suppose that $\mathbb{E}[W^4]<\infty$, which, by the moment relation (\ref{moml}), ensures that $\mathbb{E}[(W^{L})^2]<\infty$.  By Taylor expansion we have
\begin{align*}|\mathbb{E}[b^2f''(W)-f(W)]|&=b^2|\mathbb{E}[f''(W)-f''(W^L)]| \\
&\leq b^2|\mathbb{E}[f^{(3)}(W)(W-W^L)]|+\frac{b^2}{2}\|f^{(4)}\|\mathbb{E}[(W-W^L)^2] \\
&=b^2|\mathbb{E}[f^{(3)}(W)\mathbb{E}[W-W^L\,|\,W]]|+\frac{b^2}{2}\|f^{(4)}\|\mathbb{E}[(W-W^L)^2] \\
&\leq b^2\|f^{(3)}\||\mathbb{E}[|\mathbb{E}[W-W^L\,|\,W]|]+\frac{b^2}{2}\|f^{(4)}\|\mathbb{E}[(W-W^L)^2].
\end{align*}
Applying the bounds $\|f^{(3)}\|\leq \frac{1}{b}\|h''\|$ and $\|f^{(4)}\|\leq \frac{2}{b^2}\|h''\|$ from Lemma \ref{lemivp} then yields the bound (\ref{ordern}), as required.
\end{proof}

\begin{corollary}\label{corsec2}Let  $k\geq1$ and suppose that $\mathbb{E}[|W|^{k+2}]<\infty$.  Then
\begin{equation}\label{ghjk2}d_{\mathrm{K}}(W,Z)\leq 11.56\bigg(\frac{\mathbb{E}[|W-W^L|^k]}{b^k}\bigg)^{1/(k+1)}.
\end{equation}
\end{corollary}

\begin{proof}Applying Markov's inequality to (\ref{dfgh1}) gives
\begin{equation*}d_{\mathrm{K}}(W,Z)\leq \frac{(7/2+\sqrt{10})\beta}{b}+3\bigg(1+\sqrt{\frac{2}{5}}\bigg)\frac{\mathbb{E}[|W-W^L|^k]}{\beta^k},
\end{equation*}
whence on setting $\beta=(b\mathbb{E}[|W-W^L|^k])^{1/(k+1)}$ we obtain (\ref{ghjk2}).
\end{proof}

\section{Proof of Theorem \ref{thm555}}\label{sec3}

We begin by proving the following general theorem, which improves on Theorem 4.4 of \cite{pike} and Theorem 5.9 of \cite{gaunt vgii}.  The improvement comes from smaller constants than in both of those theorems and by giving the bounds in metrics stronger than the bounded Wasserstein metric bounds of \cite{pike}. Very recently, \cite{prr19} have obtained an optimal order Wasserstein distance bound for a multivariate generalisation of the following theorem.  In their result $X_1,X_2,\ldots$ are i.i.d$.$ random vectors, the limiting distribution is a centered multivariate symmetric Laplace distribution (see \cite{kkp01}) and an explicit constant is not given in their bound.




\begin{theorem}\label{thmnbc}Suppose that $X_1,X_2,\ldots$ is a sequence of independent random variables, with $\mathbb{E}[X_i]=0$ and $\mathbb{E}[X_i^2]=\sigma_i^2\in(0,\infty)$. Let $N$ be a positive, integer-valued random variable with finite mean $\mu$, which is independent of the $X_i$.  Define $\sigma^2=\frac{1}{\mu}\mathbb{E}\big[\big(\sum_{i=1}^NX_i\big)^2\big]=\frac{1}{\mu}\mathbb{E}\big[\sum_{i=1}^N\sigma_i^2\big]$.  Also, let $M$ be a random variable satisfying
\begin{equation*}\mathbb{P}(M=m)=\frac{\sigma_m^2}{\mu\sigma^2}\mathbb{P}(N\geq m), \quad m=1,2,\ldots.
\end{equation*}
Define $W_\mu=\frac{1}{\sqrt{\mu}}\sum_{i=1}^NX_i$ and let $Z\sim \mathrm{Laplace}(0,\frac{\sigma}{\sqrt{2}})$.  Then
\begin{equation}\label{wedf}d_{\mathrm{W}}(W_\mu, Z)\leq 2\mu^{-1/2}\big\{\mathbb{E}|X_M-X_M^L|+\sup_{i\geq1}\sigma_i\mathbb{E}\big[|N-M|^{\frac{1}{2}}\big]\big\}.
\end{equation}
Now suppose that $|X_i|\leq C$ for all $i$ and $|N-M|\leq K$.  Then we have
\begin{equation}\label{wdv}
d_{\mathrm{K}}(W_\mu, Z)\leq \frac{\sqrt{2}(7/2+\sqrt{10})}{\sigma\sqrt{\mu}}\Big\{\sup_{i\geq1}\|F_{X_i}^{-1}-F_{X_i^L}^{-1}\|+CK\Big\},
\end{equation}
and if $K=0$ the bound also holds for unbounded $X_i$.
\end{theorem}

\begin{proof}It was shown in the proof of Theorem 4.4 of \cite{pike} that $W_\mu^L=\mu^{-1/2}\big(\sum_{i=1}^{M-1}X_i+X_M^L\big)$. We take $X_m^L$ to be independent of $M$, $N$, and $X_k$ for all $k$.  Therefore
\begin{equation*}W_\mu^L-W_\mu=\mu^{-1/2}\bigg\{(X_M^L-X_M)+\mathrm{sgn}(M-N)\sum_{i=(M\wedge N)+1}^{N\vee M}X_i\bigg\}.
\end{equation*}
Substituting into (\ref{zezozr}) and bounding $\mathbb{E}\big|\sum_{i=(M\wedge N)+1}^{N\vee M}X_i\big|\leq \sup_{i\geq1}\sigma_i\mathbb{E}\big[|N-M|^{1/2}\big]$ (see the proof of Theorem 4.4 of \cite{pike}) gives us (\ref{wedf}).  Recall from (\ref{dfgh1}) that
\begin{align}\label{wdvb}d_{\mathrm{K}}(W_\mu,Z)\leq \frac{(7/2+\sqrt{10})\beta}{b}+\frac{15+3\sqrt{10}}{5}\mathbb{P}(|W_\mu-W_\mu^{L}|>\beta).
\end{align}
On setting $\beta=\mu^{-1/2}\big\{\sup_{i\geq1}\|F_{X_i}^{-1}-F_{X_i^L}^{-1}\|+CK\big\}$, and using Strassen's theorem we deduce (\ref{wdv}) from (\ref{wdvb}) (recalling that $b=\frac{\sigma}{\sqrt{2}}$). The assertion after inequality (\ref{wdv}) follows similarly.
\end{proof}

\noindent{\emph{Proof of Theorem \ref{thm555}.}} To ease notation, in this proof we drop the subscripts from $S_p$ and $N_p$. As noted by \cite{pike}, the assumptions imposed on $N$ and the $X_i$ imply that $\mathcal{L}(M)=\mathcal{L}(N)$, meaning that we can take $M=N$.  Inequality (\ref{wedfg}) now follows from inequality (\ref{wdv}).  To prove inequality (\ref{rwrwa}), we note the following simple inequality (see \cite{pike}) 
\begin{equation*}\mathbb{E}|X_N-X_N^L|\leq \sup_{i\geq1}\mathbb{E}|X_i|+\sup_{i\geq1}\mathbb{E}|X_i^L|=\sup_{i\geq1}\mathbb{E}|X_i|+\sup_{i\geq1}\frac{\mathbb{E}[|X_i|^3]}{3\sigma^2}\leq\sigma+\frac{\rho_3}{3\sigma^2},
\end{equation*}
where in the final step the Cauchy-Schwarz inequality was applied.  We are now able to obtain (\ref{rwrwa}) from (\ref{wedf}).

To prove inequality (\ref{taubound}), we apply inequality (\ref{ghjk2}) of Corollary \ref{corsec2}. We use the assumption that $\sup_{i\geq1}\mathbb{E}[X_i^{k+2}]<\infty$, the moment relation (\ref{moml}) and the simple inequality $|a+b|^r\leq 2^{r-1}(|a|^r+|b|^r)$, $r\geq1$, to obtain the bound
\begin{align}\mathbb{E}[|S-S^L|^k]&=p^{k/2}\mathbb{E}[|X_N-X_N^L|^k]\nonumber\\
&\leq 2^{k-1}p^{k/2}(\mathbb{E}[|X_N|^k]+\mathbb{E}[|X_N^L|^k]) \nonumber\\
\label{bvc61}&\leq 2^{k-1}p^{k/2}\bigg(\rho_k+\frac{2\rho_{k+2}}{(k+1)(k+2)\sigma^2}\bigg).
\end{align}
Substituting into (\ref{ghjk2}) then yields inequality (\ref{taubound}).

We end by establishing inequality (\ref{on11}). We now assume that $X_1,X_2,\ldots$ are identically distributed with $\mathbb{E}[X_1^3]=0$ and $\mathbb{E}[X_1^4]<\infty$.  We prove inequality (\ref{on11}) by applying inequality (\ref{ordern}) of Theorem \ref{jazzz}.  We proceed similarly to we did in obtaining (\ref{bvc61}), but this time use the independence of $X_N$ and $X_N^L$ to obtain
\begin{align}\label{taui}\mathbb{E}[(S-S^L)^2]&=p\mathbb{E}[(X_N-X_N^L)^2]=p(\mathbb{E}[X_N^2]+\mathbb{E}[(X_N^L)^2]) 
= p\bigg(\sigma^2+\frac{\mathbb{E}[X_1^4]}{6\sigma^2}\bigg).
\end{align}
We now bound $\mathbb{E}[|\mathbb{E}[S-S^L\,|\,S]|]$.  We have
\begin{align*}\mathbb{E}[S-S^L\,|\,S]=\sqrt{p}\mathbb{E}[X_N-X_N^L\,|\,S]=\sqrt{p}\big(\mathbb{E}[X_N\,|\,S]-\mathbb{E}[X_N^L]\big),
\end{align*}
as $X_N^L$ and $S$ are independent.  Also, due to the assumption that $\mathbb{E}[X_i^3]=0$ for all $i\geq1$, we have, by (\ref{moml}), that $\mathbb{E}[X_N^L]=\frac{1}{3\sigma^2}\mathbb{E}[X_N^3]=0$.  By the tower property of conditional expectation  we then have
\begin{align*}\mathbb{E}[S-S^L\,|\,S]&=\sqrt{p}\mathbb{E}[\mathbb{E}[X_N\,|S,N]\,|\,S] \\
&=\mathbb{E}\bigg[\frac{S}{N}\,\Big|\,S\bigg],
\end{align*}
where we used that because the $X_i$ are i.i.d$.$, and therefore exchangeable, $\mathbb{E}[X_N\,|\,S,N]=S/(\sqrt{p}N)$.  Therefore
\begin{align}\mathbb{E}[|\mathbb{E}[S-S^L\,|\,S]|]&=\mathbb{E}\bigg[\bigg|\mathbb{E}\bigg[\frac{S}{N}\,\Big|\,S\bigg]\bigg|\bigg]\nonumber\\
& \leq \mathbb{E}\bigg[\mathbb{E}\bigg[\frac{|S|}{N}\,\Big|\,S\bigg]\bigg]\nonumber\\
\label{pnfor}&=\mathbb{E}\bigg[\frac{|S|}{N}\bigg] 
=\sqrt{p}\sum_{n=1}^\infty \frac{1}{n}\mathbb{E}\bigg|\sum_{i=1}^n X_i\bigg|\mathbb{P}(N=n).
\end{align}
Taking $h(x)=|x|$ in inequality (\ref{minrev}) (note that $h\in\mathcal{H}_{\mathrm{W}}$) gives the inequality
\begin{equation*}\bigg|\frac{1}{\sqrt{n}}\mathbb{E}\bigg|\sum_{i=1}^n X_i\bigg|-\sqrt{\frac{2}{\pi}}\sigma\bigg|\leq \frac{\sigma}{\sqrt{n}}\bigg(2+\frac{\mathbb{E}[|X_1|^3]}{\sigma^3}\bigg)
\end{equation*}
(see \cite{cs05} for a similar bound), and on applying this inequality to (\ref{pnfor}) we obtain the bound
\begin{align}\label{n12b}\mathbb{E}[|\mathbb{E}[S-S^L\,|\,S]|]\leq \sqrt{\frac{2p}{\pi}}\sigma\mathbb{E}[N^{-1/2}]+\sqrt{p}\sigma\bigg(2+\frac{\mathbb{E}[|X_1|^3]}{\sigma^3}\bigg)\mathbb{E}[N^{-1}].
\end{align}
The expectation $\mathbb{E}[N^{-1}]$ is easily evaluated:
\begin{align*}\mathbb{E}[N^{-1}]=\sum_{n=1}^\infty\frac{p(1-p)^{n-1}}{n}=\frac{p\log(1/p)}{1-p}.
\end{align*}
We can bound $\mathbb{E}[N^{-1/2}]$ through an application of the integral test:
\begin{align*}\frac{1-p}{p}\mathbb{E}[N^{-1/2}]&=\sum_{n=1}^\infty\frac{(1-p)^n}{\sqrt{n}}<\int_0^\infty\frac{(1-p)^x}{\sqrt{x}}\,\mathrm{d}x=\int_0^\infty\frac{\exp(x\log(1-p))}{\sqrt{x}}\,\mathrm{d}x \\
&=\sqrt{\frac{2}{-\log(1-p)}}\int_0^\infty \mathrm{e}^{-t^2/2}\,\mathrm{d}t=\sqrt{\frac{\pi}{-\log(1-p)}}<\sqrt{\frac{\pi}{p}},
\end{align*}
where we used the standard inequality $\log(1+x)<x$, for $x>-1$, in the last step. Plugging the estimates for $\mathbb{E}[N^{-1/2}]$ and $\mathbb{E}[N^{-1}]$ into (\ref{n12b}) then yields the bound
\begin{align}\label{bvc5}\mathbb{E}[|\mathbb{E}[S-S^L\,|\,S]|]<\frac{\sqrt{2}\sigma p}{1-p}+\frac{\sigma p^{3/2}\log(1/p)}{1-p}\bigg(2+\frac{\mathbb{E}[|X_1|^3]}{\sigma^3}\bigg).
\end{align}
Finally, inserting (\ref{taui}) and inequality (\ref{bvc5}) into (\ref{ordern}) yields the desired bound. \hfill $\Box$

\section{Proof of Theorem \ref{thm888}}\label{sec4}

Let $Z\sim\mathrm{Laplace}(0,\frac{\sigma}{\sqrt{2}})$ and recall that $T_n=B_{n-1}^{1/2}\sum_{i=1}^nX_i$, where the $X_1,\ldots,X_n$ are independent random variables with zero mean and variance $\sigma^2\in(0,\infty)$.  Then we have the representations
\begin{align*}T_n&=_dU_nV_n, \\
Z&=_dUV,
\end{align*}
where $U_n=\sqrt{nB_{n-1}}$, $V_n=\frac{1}{\sqrt{n}}\sum_{i=1}^nX_i$, $U$ follows the Rayleigh distribution with density function $f_U(x)=2x\mathrm{e}^{-x^2}$, $x>0$, and $V\sim N(0,\sigma^2)$ are mutually independent random variables.  This representation of the Laplace distribution is given in \cite[Proposition 2.2.1]{kkp01}.  In the limit $n\rightarrow\infty$, $U_n$ converges in distribution to $U$, and, by the central limit theorem, $V_n$ converges in distribution to $V$.  Indeed, $\mathbb{P}(U_n\leq u)=1-(1-u^2/n)^{n-1}$, $u\in(0,\sqrt{n})$, which converges to $1-\mathrm{e}^{-u^2}$ as $n\rightarrow\infty$.  We prove Theorem \ref{thm888} by obtaining explicit bounds on the distance between the distributions of $U_n$ and $U$ and the distributions of $V_n$ and $V$ with respect to suitable probability metrics and then combine these bounds to bound the distance between $\mathcal{L}(T_n)$ and the $\mathrm{Laplace}(0,\frac{\sigma}{\sqrt{2}})$ distribution. We combine these bounds through the following lemma.

\begin{lemma}\label{uvuv}Let $Y_1,Y_2,Z_1,Z_2$ be real-valued random variables. Then
\begin{eqnarray}d_{\mathrm{K}}(Y_1Z_1,Y_2Z_2)&\leq& d_{\mathrm{K}}(Y_1,Y_2)+d_{\mathrm{K}}(Z_1,Z_2), \nonumber\\
d_{\mathrm{W}}(Y_1Z_1,Y_2Z_2)&\leq& \mathbb{E}|Z_1|d_{\mathrm{W}} (Y_1,Y_2)+\mathbb{E}|Y_2|d_{\mathrm{W}}(Z_1,Z_2),\nonumber \\
\label{d12d12}d_{1,2}(Y_1Z_1,Y_2Z_2)&\leq& \mathbb{E}|Z_1|d_{\mathrm{W}} (Y_1,Y_2)+\mathbb{E}[Y_2^2]d_{2}(Z_1,Z_2),
\end{eqnarray}
where each inequality holds provided the expectations in the the right-hand side of the inequality exist.
\end{lemma}

\begin{proof}We prove the bound for $d_{1,2}$; the bounds for $d_{\mathrm{K}}$ and $d_{\mathrm{W}}$ are obtained through similar and slightly simpler arguments.  Let $h\in\mathcal{H}_{1,2}$.  Then, by the triangle inequality and conditioning,
\begin{align}&|\mathbb{E}[h(Y_1Z_1)]-\mathbb{E}[h(Y_2Z_2)]| \nonumber\\
&\leq |\mathbb{E}[h(Y_1Z_1)]-\mathbb{E}[h(Y_2Z_1)]|+|\mathbb{E}[h(Y_2Z_1)]-\mathbb{E}[h(Y_2Z_2)]|\nonumber \\
&=|\mathbb{E}[\mathbb{E}[h(Y_1Z_1)-h(Y_2Z_1)]\,|\,Z_1]]|+|\mathbb{E}[\mathbb{E}[h(Y_2Z_1)-h(Y_2Z_2)]\,|\,Y_2]]| \nonumber\\
\label{polkj}&\leq \mathbb{E}[|\mathbb{E}[h(Y_1Z_1)-h(Y_2Z_1)]\,|\,Z_1]|]+\mathbb{E}[|\mathbb{E}[h(Y_2Z_1)-h(Y_2Z_2)]\,|\,Y_2]|].
\end{align}
Now, for $a\in\mathbb{R}\setminus\{0\}$ and real-valued random variables $X$ and $Y$ we have that 
\begin{eqnarray*}|\mathbb{E}[h(aX)]-\mathbb{E}[h(aY)]|&\leq& d_{\mathrm{W}}(aX,aY)=ad_{\mathrm{W}}(X,Y), \\
|\mathbb{E}[h(aX)]-\mathbb{E}[h(aY)]|&\leq& d_{2}(aX,aY)=a^2d_{2}(X,Y), 
\end{eqnarray*} 
since $\mathcal{H}_{1,2}\subset \mathcal{H}_{\mathrm{W}}$ and $\mathcal{H}_{1,2}\subset \mathcal{H}_{2}$.  Applying these inequalities to (\ref{polkj}) we obtain that, for $h\in\mathcal{H}_{1,2}$,
\begin{align}|\mathbb{E}[h(Y_1Z_1)]-\mathbb{E}[h(Y_2Z_2)]|&\leq\mathbb{E}[|Z_1d_{\mathrm{W}} (Y_1,Y_2)
|]+\mathbb{E}[|Y_2^2d_{\mathrm{2}} (Z_1,Z_2)
|]\nonumber\\
\label{polk}&=\mathbb{E}|Z_1|d_{\mathrm{W}} (Y_1,Y_2)+\mathbb{E}[Y_2^2]d_{2}(Z_1,Z_2)
\end{align}
The bound (\ref{polk}) holds for all $h\in\mathcal{H}_{1,2}$, and as $d_{1,2}(Y_1Z_1,Y_2Z_2)=\sup_{h\in\mathcal{H}_{1,2}}|\mathbb{E}[h(Y_1Z_1)]-\mathbb{E}[h(Y_2Z_2)]|$ it follows that inequality (\ref{d12d12}) holds.
\end{proof}

There is a vast literature on bounds for $d_{\mathcal{H}}(V_n,V)$.  We will make use of three bounds from the literature for the cases $\mathcal{H}_{\mathrm{K}}$, $\mathcal{H}_{\mathrm{W}}$ and $\mathcal{H}_2$.

\begin{theorem}[Shevtsova \cite{s10}] \label{thmapk} Let $X_1,\ldots,X_n$ be independent random variables with $\mathbb{E}[X_i]=0$, $\mathrm{Var}(X_i)=\sigma^2\in(0,\infty)$ and $\mathbb{E}[|X_i|^3]<\infty$, for all $1\leq i\leq n$.  Denote $V_n=\frac{1}{\sqrt{n}}\sum_{i=1}^nX_i$ and let $V\sim N(0,\sigma^2)$.  Then
\begin{equation*}d_{\mathrm{K}}(V_n,V)\leq\frac{C_0}{\sigma^3n^{3/2}}\sum_{i=1}^n\mathbb{E}[|X_i|^3],
\end{equation*}
where $C_0=0.5600$.
\end{theorem}

\begin{theorem}[Reinert \cite{r98}] \label{thmapw} Under the same assumptions as Theorem \ref{thmapk}, we have that, for $h\in\mathcal{H}_{\mathrm{W}}$,
\begin{equation}\label{minrev}|\mathbb{E}[h(V_n)]-\mathbb{E}[h(V)]|\leq\frac{\sigma}{n^{3/2}}\sum_{i=1}^n\bigg(2+\frac{\mathbb{E}[|X_i|^3]}{\sigma^3}\bigg).
\end{equation}
Consequently,
\begin{equation}\label{thmapwd}d_{\mathrm{W}}(V_n,V)\leq\frac{\sigma}{n^{3/2}}\sum_{i=1}^n\bigg(2+\frac{\mathbb{E}[|X_i|^3]}{\sigma^3}\bigg).
\end{equation}
\end{theorem}

\begin{theorem}[Gaunt \cite{gaunt rate}] \label{thmap2} Let $X_1,\ldots,X_n$ be independent random variables with $\mathbb{E}[X_i]=0$, $\mathrm{Var}(X_i)=\sigma^2\in(0,\infty)$, $\mathbb{E}[X_i^3]=0$ and $\mathbb{E}[X_i^4]<\infty$, for all $1\leq i\leq n$. Then
\begin{equation}\label{thmap2d}d_{2}(V_n,V)\leq\frac{\sigma^2}{n^2}\sum_{i=1}^n\bigg(1+\frac{\mathbb{E}[X_i^4]}{3\sigma^4}\bigg).
\end{equation}
\end{theorem}

\begin{remark}The Berry-Esseen Theorem \ref{thmapk}, with a larger constant $C_0$, was proved independently by Berry \cite{berry} and Esseen \cite{esseen} in the early 1940s, and since then several works have improved on the constant with the best estimate of $C_0=0.5600$ due to \cite{s10}.  For i.i.d$.$ random variables $X_1,\ldots,X_n$, the constant improves to $C_0=0.4748$ \cite{s11}.  The assumption of bounded third absolute moments can also be reduced at the expense of a slightly more complicated bound with bigger constants \cite{feller}.  Theorem \ref{thmapw} is formulated slightly differently in Theorem 2.1 of \cite{r98}, but by re-scaling we obtain the bound (\ref{thmapwd}).  This is also the case for Theorem \ref{thmap2}, and we additionally obtain an improved constant in (\ref{thmap2d}) by using the bound $\|f^{(4)}\|\leq 2\|h''\|$ (due to \cite{daly}) for the solution of the standard normal Stein equation $f''(x)-xf'(x)=h(x)-\mathbb{E}[N]$, $N\sim N(0,1)$, rather than the bound $\|f^{(4)}\|\leq3\|h''\|$ that was used in proof of Theorem 3.1 of \cite{gaunt rate}.
\end{remark}

As the Rayleigh distribution is a special case of the generalized gamma distribution, the following lemma follows as a special case of Proposition 2.3 of \cite{gaunt ngb}.

\begin{lemma}\label{lemray}Let $U$ denote a Rayleigh random variable with probability density function
$p_U(x)=2x\mathrm{e}^{-x^2}$, $x>0$. Suppose that $f:(0,\infty)\rightarrow\mathbb{R}$ is differentiable and such that $\mathbb{E}|Uf'(U)|<\infty$, $\mathbb{E}|f(U)|<\infty$ and $\mathbb{E}|U^2f(U)|<\infty$.  Then
\begin{equation*}\mathbb{E}[\mathcal{A}_Uf(U)]=0,
\end{equation*}
where $\mathcal{A}_Uf(x)=xf'(x)+(2-2x^2)f(x)$.
\end{lemma}

\begin{lemma}\label{lembet}Let $U_n=\sqrt{n B_{n-1}}$, where $B_{n-1}\sim \mathrm{Beta}(1,n-1)$.  Suppose that $f:(0,\sqrt{n})\rightarrow\mathbb{R}$ is differentiable and such that $\mathbb{E}|U_nf'(U_n)|<\infty$, $\mathbb{E}|U_n^3f'(U_n)|<\infty$, $\mathbb{E}|f(U_n)|<\infty$ and $\mathbb{E}|U_n^2f(U_n)|<\infty$.  Then
\begin{align}\label{beta12}\mathbb{E}[\mathcal{A}_{U_n}f(U_n)]=0.
\end{align}
where $\mathcal{A}_{U_n}f(x)=x(1-x^2/n)f'(x)+(2-2x^2)f(x)$.
\end{lemma}

\begin{proof}Define the operator $T_r$ by $T_ry(x)=xy'(x)+ry(x)$, $r\in\mathbb{R}$.  In this notation, the classical Stein operator for the $\mathrm{Beta}(1,n-1)$ distribution is given by $A_{B_{n-1}}y(x)=T_1y(x)-xT_{n}y(x)$ \cite{dobler beta, goldstein4}.  Let $C_{n}=B_{n-1}^{1/2}$ and let $g:(0,1)\rightarrow\mathbb{R}$ by such that $\mathbb{E}|C_{n}g'(C_{n})|<\infty$, $\mathbb{E}|C_{n}^3g'(C_{n})|<\infty$, $\mathbb{E}|g(C_{n})|<\infty$ and $\mathbb{E}|C_{n}^2g(C_{n})|<\infty$.  Then, by equation (15) of \cite{GMS16},
\begin{equation}\label{hold}\mathbb{E}[T_2g(C_{n})-C_{n}^2T_{2n}g(C_{n})]=0.
\end{equation}
(The conditions on $g$ that are stated above are not specified in \cite{GMS16}, but on examining their analysis one can see that these conditions ensure that (\ref{hold}) holds.) That is 
\begin{equation}\label{beta123}\mathbb{E}\big[C_n(1-C_n^2)g'(C_n)+(2n-2nC_n^2)g(C_n)\big]=0.
\end{equation}
We have that $U_n=_d\sqrt{n}C_n$, and on rescaling we deduce (\ref{beta12}) from (\ref{beta123}).
\end{proof}

In the following lemma, the bound (\ref{pl12}) is proved purely for reasons of exposition, as an improved bound will be stated in Remark \ref{yvik}.  Proving both the Kolmogorov and Wasserstein distance bounds requires very little more work than only proving the Wasserstein distance bound.

\begin{lemma}\label{tper}Let the random variables $U_n$ and $U$ be defined as above.  Then, for $n\geq2$,
\begin{eqnarray}\label{pl12}d_{\mathrm{K}}(U_n,U)&\leq&\frac{2}{n}, \\
\label{pl14}d_{\mathrm{W}}(U_n,U)&\leq&\frac{11.49}{n}.
\end{eqnarray}
\end{lemma}

\begin{proof}Let the Stein operators $A_U$ and $A_{U_n}$ be defined as in Lemmas \ref{lemray} and \ref{lembet}, respectively.  Suppose that $h:(0,\infty)\rightarrow\mathbb{R}$ is either bounded or Lipschitz.  Let $f$ be the solution of the $\mathrm{Rayleigh}(1/\sqrt{2})$ Stein equation $A_Uf(x)=h(x)-\mathbb{E}[h(U)]$, which by Lemma \ref{propapp3}, we know satisfies the bounds
\begin{align}\label{pl78}\|xf'(x)\|&\leq \frac{2}{2^{-1}}\times\frac{1}{2}\|h-\mathbb{E}[h(U)]\|\leq2, \quad h\in\mathcal{H}_{\mathrm{K}}, \\
\label{pl79}\|f'\|&\leq\frac{6.11}{2^{-3/2}}\times \frac{1}{2} \|h'\|\leq8.6408, \quad h\in\mathcal{H}_{\mathrm{W}}.
\end{align}
  Then 
\begin{align}|\mathbb{E}[h(U_n)]-\mathbb{E}[h(U)]|&=|\mathbb{E}[\mathcal{A}_{U}f(U_n)]|=|\mathbb{E}[\mathcal{A}_{U}f(U_n)-\mathcal{A}_{U_n}f(U_n)]| \nonumber \\
&=\frac{1}{n}|\mathbb{E}[U_n^3f'(U_n)] | \nonumber\\
\label{near3}&\leq \frac{1}{n}\min\Big\{\|xf'(x)\|\mathbb{E}[U_n^2],\|f'\|\mathbb{E}[U_n^3]\Big\}.
\end{align}
That $\mathbb{E}[\mathcal{A}_{U_n}f(U_n)]=0$ follows from the assumptions on $h$ and the estimates of Lemma \ref{propapp3} for the solution of the Rayleigh Stein equation.   Now, $\mathbb{E}[U_n^2]=1$ and  
\begin{align}\mathbb{E}[U_n^{3}]&=n^{3/2}\mathbb{E}[B_{n-1}^{3/2}]=n^{3/2}\int_0^1(n-1)x^{3/2}(1-x)^{n-2}\,\mathrm{d}t=n^{3/2}(n-1) B\big(\tfrac{5}{2},n-1\big)\nonumber\\
\label{un3un}&=n^{3/2}(n-1)\frac{\Gamma(5/2)\Gamma(n-1)}{\Gamma(n+3/2)}=\frac{3\sqrt{\pi}n^{3/2}\Gamma(n)}{4\Gamma(n+3/2)},
\end{align}
where $B(a,b)=\int_0^1 x^{a-1}(1-x)^{b-1}\,\mathrm{d}x$ is the beta function, and we used the standard formulas $u\Gamma(u)=\Gamma(u+1)$ and $\Gamma(5/2)=3\sqrt{\pi}/4$.  Now $n^{3/2}\Gamma(n)/\Gamma(n+3/2)$ is an increasing function of $n$ on $(0,\infty)$ \cite{ismail}.  Therefore, for $n\geq2$,
\begin{align*}\mathbb{E}[U_n^{3}]\leq\frac{3\sqrt{\pi}}{4}\lim_{n\rightarrow\infty}\frac{n^{3/2}\Gamma(n)}{\Gamma(n+3/2)}=\frac{3\sqrt{\pi}}{4},
\end{align*}
where the limit follows from \cite[formula (5.6.4)]{olver}.   Applying the bounds (\ref{pl78}) and (\ref{pl79}) together with the bounds for $\mathbb{E}[U_n^2]$ and $\mathbb{E}[U_n^3]$ to (\ref{near3}) then yields the bounds (\ref{pl12}) and (\ref{pl14}).
\end{proof}

\begin{remark}\label{yvik}The following bounds will appear in the supplementary material of the arXiv version of the preprint \cite{es19}.  For $n\geq2$,
\begin{align}\label{eskol}d_{\mathrm{K}}(U_n,U)\leq \frac{1}{n}\bigg(1+2\bigg(1-\frac{2}{n}\bigg)^{n-2}\bigg),
\end{align}
and
\begin{align}&d_{\mathrm{W}}(U_n,U)\leq-\frac{\sqrt{\pi}\Gamma(n)}{4\sqrt{n}\Gamma(n+1/2)}\nonumber\\
\label{esw}&\quad+2\sqrt{2}\frac{n-1}{n^n}\cdot\frac{(n-2)^nn(40+11(n-4)n)+(n-2)^3n^n {}_2F_1(-\frac{1}{2},3-n;\frac{1}{2};\frac{2}{n})}{(n-2)^2(2n-5)(2n-3)(2n-1)},
\end{align}
where ${}_2F_1(a,b;c;x)$ is the Gaussian hypergeometric function. (We define $0^0:=1$, but this is irrelevant because the bound (\ref{eskol}) is greater than 1 in this case.) These bounds were obtained using a recent technique of \cite{es19} for bounding distances between distributions that builds upon the formalism of \cite{ers19} for new representations of solutions to Stein equations.  For another recent approach to bounding distances between distributions, see \cite{dsw19}.

  Our Kolmogorov distance bound (\ref{pl12}) outperforms (\ref{eskol}) when $n=2$ (although in this case the upper bound of 1 is trivial), but for all $n\geq3$ the reverse is true.  Numerical calculations carried using \emph{Mathematica} suggest that the Wasserstein bound (\ref{esw}) improves on our bound (\ref{pl14}) for all $n\geq2$, although verifying this assertion analytically seems to be difficult.  Our bound is of course much simpler and the dependence on $n$ is very clear.  For this reason, we will use the bound (\ref{pl14}) in our proof of Theorem \ref{thm888}.
\end{remark}


\noindent{\emph{Proof of Theorem \ref{thm888}.}} Recall that $T_n=_d U_nV_n$ and $Z=_dUV$.  Then, by Lemma \ref{uvuv},
\begin{eqnarray}\label{d12d127}d_{\mathrm{K}}(T_n,Z)&\leq& d_{\mathrm{K}}(U_n,U)+d_{\mathrm{K}}(V_n,V), \\
\label{d12d126}d_{\mathrm{W}}(T_n,Z)&\leq& \mathbb{E}|V|d_{\mathrm{W}} (U_n,U)+\mathbb{E}[U_n]d_{\mathrm{W}}(V_n,V), \\
\label{d12d125}d_{1,2}(T_n,Z)&\leq& \mathbb{E}|V|d_{\mathrm{W}} (U_n,U)+\mathbb{E}[U_n^2]d_{2}(V_n,V).
\end{eqnarray}
By standard formulas for the moments and absolute moments of the beta and normal distributions, we have that $\mathbb{E}[U_n^2]=1$ and $\mathbb{E}|V|=\sigma\sqrt{2/\pi}$.  Also, by a similar calculation to the one used to obtain the formula (\ref{un3un}) we have, for $n\geq2$,
\begin{align*}\mathbb{E}[U_n]=\frac{\sqrt{\pi}\sqrt{n}\Gamma(n)}{2\Gamma(n+1/2)}\leq\frac{\sqrt{\pi}\sqrt{2}\Gamma(2)}{2\Gamma(5/2)}=\frac{2\sqrt{2}}{3},
\end{align*}
where we used that $\sqrt{n}\Gamma(n)/\Gamma(n+1/2)$ is a decreasing function of $n$ on $(0,\infty)$ \cite{ismail}.  Theorems \ref{thmapk} -- \ref{thmap2} give bounds for $d_{\mathrm{K}}(V_n,V)$, $d_{\mathrm{W}}(V_n,V)$ and $d_{2}(V_n,V)$, respectively, and $d_{\mathrm{K}}(U_n,U)$ is bounded by inequality (\ref{eskol}) and $d_{\mathrm{W}}(U_n,U)$ is bounded by inequality (\ref{pl14}).  Substituting all of these estimates into (\ref{d12d127}), (\ref{d12d126}) and (\ref{d12d125}) then yields the bounds as stated in Theorem \ref{thm888}. \hfill $\Box$

\section{The Rayleigh Stein equation}\label{sec5}

Let $R\sim \mathrm{Rayleigh}(\sigma)$, $\sigma>0$, follow the Rayleigh distribution with density function
\[\rho_R(x)=\frac{x}{\sigma^2}\mathrm{e}^{-x^2/(2\sigma^2)}, \quad x>0.\]
The Rayleigh distribution is a special case of the chi distribution (up to scaling).  A random variable $K$ following the chi distribution with $k>0$ degrees of freedom, denoted by $\chi_{(k)}$, has probability density function
\begin{equation*}\rho_k(x)=\frac{1}{2^{k/2-1}\Gamma(k/2)}x^{k-1}\mathrm{e}^{-x^2/2}, \quad x>0.
\end{equation*}
We proceed by obtaining bounds for the solution of the chi distribution Stein equation, before specialising to the solution of the Rayleigh Stein equation.

We first note that the density $\rho_k$ satisfies the differential equation
\begin{equation}\label{pode}\big(s(x)\rho(x)\big)'=\tau(x)\rho(x),
\end{equation}
where $s(x)=x$ and $\tau(x)=k-x^2$.  It therefore follows from Theorem 1 of \cite{schoutens} that a Stein equation for the $\chi_{(k)}$ distribution is given by
\begin{equation}\label{chisteineqn}xf'(x)+(k-x^2)f(x)=h(x)-\mathbb{E} [h(K)],
\end{equation}
where $K\sim\chi_{(k)}$.  It is straightforward to solve (\ref{chisteineqn}) (see Proposition 1 of \cite{schoutens}):
\begin{align}\label{chisoln}f(x)&=\frac{1}{x\rho_k(x)}\int_0^x (h(t)-\mathbb{E} [h(K)])\rho_k(t)\,\mathrm{d}t, \\
\label{chisoln2}&=-\frac{1}{x\rho_k(x)}\int_x^\infty (h(t)-\mathbb{E} [h(K)])\rho_k(t)\,\mathrm{d}t.
\end{align}

In order to bound the solution (\ref{chisoln}) and its first derivative, it will be useful to note the following straightforward extension of Lemmas 1 and 3 of \cite{schoutens2}.

\begin{lemma}\label{propapp1}Let $\rho$ be the probability density function of a random variable $Y$, supported on $(a,b)$, which satisfies the differential equation (\ref{pode}), where $s(x)$ is a polynomial of degree no greater than two and $\tau(x)$ is monotonic in $(a,b)$ with exactly one sign change at the point $m\in(a,b)$. Let $h:(a,b)\rightarrow\mathbb{R}$ be bounded.  Then, the solution of the Stein equation $s(x)f'(x)+\tau(x)f(x)=h(x)-\mathbb{E} h(Y)$, as given by $f(x)=\frac{1}{s(x)\rho(x)}\int_a^x(h(t)-\mathbb{E} [h(Y)])\rho(t)\,\mathrm{d}t$, satisfies the bounds
\begin{align}\label{bdd1}\|f\|&\leq M\|h-\mathbb{E} [h(Y)]\|,\\
\label{bbd2}\|s(x)f'(x)\|&\leq 2\|h-\mathbb{E} [h(Y)]\|,
\end{align}
where
\[M=\frac{1}{s(m)\rho(m)}\max\{F(m),1-F(m)\},\]
with $F$ denoting the distribution function of $Y$.
\end{lemma}

\begin{remark}The bound (\ref{bdd1}) is a generalisation of the corresponding bound of Lemma 1 of \cite{schoutens2}, which is only given for the case that $\tau(x)=a(\mathbb{E} [Y]-x)$, where $a\not=0$.  The crucial feature of this function that is exploited in the proof of \cite{schoutens2} is that $\tau(x)$ is monotonic with exactly one sign change at $x=\mathbb{E} [Y]$.  As noted by \cite{kk18}, we can therefore extend the result of \cite{schoutens2} to any $\tau(x)$ that is monotonic with only one change of sign.
\end{remark}

\begin{lemma}\label{propapp2}Let $f:(0,\infty)\rightarrow\mathbb{R}$ denote the solution (\ref{chisoln}) of the Stein equation (\ref{chisteineqn}).  Let $h:(0,\infty)\rightarrow\mathbb{R}$ be bounded.  Then
\begin{align}\label{bound10}\|f\|&\leq\frac{\Gamma(k/2)\mathrm{e}^{k/2}}{2(k/2)^{k/2}}\|h-\mathbb{E} [h(K)]\|, \\
\label{bound20}\|xf'(x)\|&\leq 2\|h-\mathbb{E} [h(K)]\|.
\end{align}
\end{lemma}

\begin{proof}Bounds (\ref{bound10}) and (\ref{bound20}) follow easily from Lemma \ref{propapp1}; note that $\tau(x)=k-x^2$ satisfies the assumption of the lemma.   To apply the lemma, we note that here $m=\sqrt{k}$, being the positive solution to the equation $k-x^2=0$; $s(x)=x$; and we used the trivial bound $\max\{F(m),1-F(m)\}\leq1$. 
\end{proof}

We now specialise to the case $k=2$, which corresponds to the Rayleigh distribution.

\begin{lemma}\label{propapp3}Let $f$ denote the solution of the Rayleigh Stein equation $\sigma^2xf'(x)+(2\sigma^2-x)f(x)=h(x)-\mathbb{E} [h(R)]$, where $R\sim\mathrm{Rayleigh}(\sigma)$. Let $h:(0,\infty)\rightarrow\mathbb{R}$ be bounded.  Then
\begin{align}\label{bound1}\|f\|&\leq \frac{\mathrm{e}}{2\sigma^2}\|h-\mathbb{E} [h(R)]\|, \\
\label{bound2}\|xf'(x)\|&\leq \frac{2}{\sigma^2}\|h-\mathbb{E} [h(R)]\|.
\end{align}
Now suppose that $h$ is Lipschitz.  Then
\begin{align}\label{bound3}\|xf(x)\|&\leq\frac{2.325}{\sigma}\|h'\|, \\
\label{bound4}\|f'\|&\leq \frac{6.11}{\sigma^3}\|h'\|, \\
\label{bound5}\|xf''(x)\|&\leq \frac{11.30}{\sigma^3}\|h'\|.
\end{align}
\end{lemma}

\begin{proof}For ease of notation, we consider the case $\sigma=1$.  The general case follows from rescaling.  Bounds (\ref{bound1}) and (\ref{bound2}) follow immediately from Lemma \ref{propapp2}. 

Now we prove inequality (\ref{bound3}).  Let $h$ be Lipschitz.  By the mean value theorem, for $t>0$, $|h(t)-\mathbb{E} [h(R)]|\leq \|h'\|(t+\mathbb{E} [R])=\|h'\|(t+\sqrt{\pi/2})$.  Therefore, for $x>0$,
\begin{align*}|xf(x)|\leq\frac{\|h'\|}{\rho_R(x)}\int_0^x\big(\sqrt{\tfrac{\pi}{2}}+t\big)\rho_R(t)\,\mathrm{d}t=:\frac{\|h'\|}{\rho_R(x)}I_1(x), 
\end{align*}
and
\begin{align*}|xf(x)|\leq\frac{\|h'\|}{\rho_R(x)}\int_x^\infty\big(\sqrt{\tfrac{\pi}{2}}+t\big)\rho_R(t)\,\mathrm{d}t=:\frac{\|h'\|}{\rho_R(x)}I_2(x).
\end{align*}
By integration by parts, the integrals $I_1(x)$ and $I_2(x)$ can  be evaluated in terms of the error function $\mathrm{erf}(x)=\frac{2}{\sqrt{\pi}}\int_0^x \mathrm{e}^{-t^2}\,\mathrm{d}t$:
\begin{align*}I_1(x)&=\sqrt{\frac{\pi}{2}}(1-\mathrm{e}^{-x^2/2})+\sqrt{\frac{\pi}{2}}\mathrm{erf}\bigg(\frac{x}{\sqrt{2}}\bigg)-x\mathrm{e}^{-x^2/2}, \\
I_2(x)&=\sqrt{\frac{\pi}{2}}(1+\mathrm{e}^{-x^2/2})-\sqrt{\frac{\pi}{2}}\mathrm{erf}\bigg(\frac{x}{\sqrt{2}}\bigg)+x\mathrm{e}^{-x^2/2}.
\end{align*}
It can be seen that $I_1(x)/\rho_R(x)$ and $I_2(x)/\rho_R(x)$ are increasing and decreasing functions of $x$, respectively, and we used \emph{Mathematica} to compute that the two functions intersect at the point $x^*=1.360722\ldots$.  Therefore, for all $x>0$,
\begin{align*}|xf(x)|\leq\frac{I_1(x^*)}{p(x^*)}\|h'\|=2.325\|h'\|.
\end{align*}

Lastly, we establish the bounds (\ref{bound4}) and (\ref{bound5}). Differentiating both sides of (\ref{chisteineqn}) and rearranging gives 
\begin{equation}\label{chi333}xf''(x)+(3-x^2)f'(x)=h'(x)+2xf(x),
\end{equation}  
which we recognise as the $\chi_{(3)}$ Stein equation with test function $h'(x)+2xf(x)$, applied to the function $f'$.  It is important to note that the test function $h'(x)+2xf(x)$ has mean zero with respect to the random variable $K_3\sim\chi_{(3)}$.  This follows because $xf''(x)+(3-x^2)f'(x)$ is a Stein operator for the $\chi_{(3)}$ distribution, meaning that $\mathbb{E}[K_3f''(Y)+(3-K_3^2)f'(K_3)]=0$, and therefore from (\ref{chi333}) we have that $\mathbb{E}[h'(K_3)+2K_3f(K_3)]=0$.  We can therefore use the iterative technique of \cite{dgv15} to deduce bounds for $\|f'\|$ and $\|xf''(x)\|$ from our bounds (\ref{bound10}) and (\ref{bound20}) with $k=3$ and (\ref{bound3}).  We have
\begin{align*}\|f'\|&\leq \frac{2\Gamma(3/2)\mathrm{e}^{3/2}}{(3/2)^{3/2}}\|h'(x)+2xf(x)\|\leq \frac{\Gamma(3/2)\mathrm{e}^{3/2}}{2(3/2)^{3/2}}\big(\|h'\|+2\|xf(x)\|\big) \\
&\leq \frac{\Gamma(3/2)\mathrm{e}^{3/2}}{2(3/2)^{3/2}}(1+2\cdot 2.325)\|h'\|=6.11\|h'\|,
\end{align*}
and
\begin{align*}\|xf''(x)\|\leq 2\|h'(x)+2xf(x)\|\leq2(1+2\cdot 2.325)\|h'\|=11.30\|h'\|,
\end{align*}
which completes the proof.
\end{proof}

\subsection*{Acknowledgements}
The author is supported by a Dame Kathleen Ollerenshaw Research Fellowship.  The author would like to thank Yvik Swan for generously sharing some results that will be added to the supplementary material of the arXiv version of his preprint \cite{es19}, which are stated in Remark \ref{yvik}. The author is grateful to the referees for their careful reading of the manuscript and for identifying several typos and errors in the displayed equations.


\footnotesize

\begin{thebibliography}{99}
\addcontentsline{toc}{section}{References}

\bibitem{ah19} Arras, B. and Houdr\'{e}, C. \emph{On Stein's method for infinitely divisible laws with finite first moment.} Springer Briefs in Probability and Mathematical Statistics. Springer, Cham, 2019.




\bibitem{berry} Berry, A. C. The Accuracy of the Gaussian Approximation to the Sum of Independent Variates. \emph{T. Am. Math. Soc.} $\mathbf{49}$ (1941), pp. 122--136.


\bibitem{bj2} Braverman, A. and Dai, J. G. High order steady-state diffusion approximation of the Erlang-C system.  arXiv:1602.02866, 2016.





\bibitem{cs05} Chen, L. H. Y. and Shao, Q.--M.  Stein's method for normal approximation. In \emph{Lect. Notes Ser. Inst. Math. Sci. Natl. Univ. Singap.} $\mathbf{4}$ (2005), pp. 1--59. Singapore Univ. Press, Singapore.


\bibitem{daly} Daly, F.  Upper bounds for Stein-type operators. \emph{Electon. J. Probab.} $\mathbf{13}$ (2008), pp. 566--587.

\bibitem{dobler beta} D\"{o}bler, C. Stein's method of exchangeable
  pairs for the beta distribution and
  generalizations. \emph{Electron. J. Probab.} $\mathbf{20}$ no$.$ 109 (2015),
  pp. 1--34.


\bibitem{dobler} D\"{o}bler, C. Distributional transformations without orthogonality relations. \emph{J. Theor. Probab.} $\mathbf{30}$ (2017), pp. 85--116.

\bibitem{dgv15} D\"{o}bler, C, Gaunt, R. E. and Vollmer, S. J.  An iterative technique for bounding derivatives of solutions of Stein equations.  \emph{Electron. J. Probab.} $\mathbf{22}$ no. 96 (2017), pp. 1--39.


\bibitem{dsw19} Duembgen, L. Samworth, R. and Wellner, J.  Bounding distributional errors via density ratios. arXiv:1905.03009, 2019.





\bibitem{ers19}Ernst, M., Reinert, G. and Swan, Y.  First order covariance inequalities via Stein's method. To appear in \emph{Bernoulli}, 2020+.

\bibitem{es19} Ernst, M. and Swan, Y. Distances between distributions via Stein's method. arXiv:1909.11518, 2019.

\bibitem{esseen}  Esseen, C. G. On the Liapunoff limit of error in the theory of probability. \emph{Ark. Mat. Astron. Fys.} (1942), Vol. A28, No. 9, pp. 1--19.

\bibitem{f18} Fathi, M. Higher-order Stein kernels for Gaussian approximation. To appear in \emph{Stud. Math.} 2020+. 

\bibitem{feller} Feller, W. On the Berry-Esseen Theorem. \emph{Z. Wahrscheinlichkeit} $\mathbf{10}$ (1968),  pp. 261--268.

\bibitem{gaunt vg} Gaunt, R. E.  Variance-Gamma approximation via Stein's method.  \emph{Electron. J. Probab.} $\mathbf{19}$ no. 38 (2014), pp. 1--33.

\bibitem{gaunt rate}  Gaunt, R. E. Rates of Convergence in Normal Approximation Under Moment Conditions Via New Bounds on Solutions of the Stein Equation.  \emph{J. Theor. Probab.} $\mathbf{29}$ (2016),  pp. 231--247.


\bibitem{gaunt ngb} Gaunt, R. E.  Products of normal, beta and gamma
  random variables: Stein operators and distributional theory. \emph{Braz. J. Probab. Stat.} $\mathbf{32}$ (2018), pp. 437--466.


\bibitem{gaunt vgii} Gaunt, R. E. Wasserstein and Kolmogorov error bounds for variance-gamma approximation via Stein's method I. \emph{J. Theor. Probab.} $\mathbf{33}$ (2020), pp. 465--505.

\bibitem{gaunt fon} Gaunt, R. E. Stein's method for functions of multivariate normal random variables. \emph{Ann. I. H. Poincare-Pr.} $\mathbf{56}$ (2020), pp. 1484--1513.

\bibitem{GMS16} Gaunt, R. E., Mijoule, G. and Swan, Y.  An algebra of Stein operators. \emph{J. Math. Anal. Appl.} $\mathbf{469}$ (2019), pp. 260--279.

\bibitem{gaunt chi square} Gaunt, R. E., Pickett, A. M. and Reinert, G.  Chi-square approximation by Stein's method with application to Pearson's statistic.  \emph{Ann. Appl. Probab.} $\mathbf{27}$ (2017), pp. 720--756.



\bibitem{goldstein} Goldstein, L. and Reinert, G.  Stein's Method and the zero bias transformation with application to simple random sampling.  \emph{Ann. Appl. Probab.} $\mathbf{7}$ (1997), pp. 935--952.

\bibitem{goldstein4} Goldstein, L. and Reinert, G.  Stein's method for
  the Beta distribution and the P\'{o}lya-Eggenberger Urn.
  \emph{J. Appl. Probab.} $\mathbf{50}$ (2013), pp. 1187--1205. 



\bibitem{ismail} Ismail, M. E. H., Lorch, L. and Muldoon, M. E. Completely monotonic functions associated with the gamma function and its q-analogues. \emph{J. Math. Anal. Appl.} $\mathbf{116}$ (1986), pp. 1--9.

\bibitem{k84} Kakosyan, A. V., Klebanov, L. B. and Melamed, I. A.  \emph{Characterization of Distributions by the Method of Intensively Monotone Operators, Lecture Notes in Math.} $\mathbf{1088}$ Springer, Berlin, 1984.

\bibitem{k97} Kalashnikov, V. \emph{Geometric Sums: Bounds for Rare Events with Applications. Risk Analysis, Reliability, Queueing.}  Kluwer Academic Publishers Group, Dordrecht, 1997.

\bibitem{kk18} Konzou, E. and Koudou, A. About the Stein equation for the generalized inverse Gaussian and Kummer distributions. \emph{ESAIM: PS} $\mathbf{24}$ (2020), pp. 607--626.

\bibitem{kkp01} Kotz, S., Kozubowski, T. J. and Podg\'{o}rski, K. \emph{The Laplace Distribution and Generalizations: A Revisit with New Applications.} Springer, 2001. 

\bibitem{lefevre} Lef\`{e}vre, C. and Utev, S.  Exact norms of a Stein-type operator and associated stochastic orderings. \emph{Probab. Theory Rel.} $\mathbf{127}$ (2003),  pp. 353--366.

\bibitem{ley} Ley, C., Reinert, G. and Swan, Y.  Stein's method for comparison of univariate distributions. \emph{Probab. Surv.} $\mathbf{14}$ (2017), pp. 1--52.

\bibitem{luk} Luk, H.  \emph{Stein's Method for the Gamma Distribution
    and Related Statistical Applications.}  PhD thesis, University of
  Southern California, 1994. 


\bibitem{olver} Olver, F. W. J., Lozier, D. W., Boisvert, R. F. and Clark, C. W.  \emph{NIST Handbook of Mathematical Functions.} Cambridge University Press, 2010.

\bibitem{pakes1} Pakes, A. G.  A characterization of gamma mixtures of stable laws motivated by limit theorems. \emph{Stat. Neerl.} $\mathbf{46}$ (1992), pp. 209--218.

\bibitem{pakes2} Pakes, A. G.  On characterizations through mixed sums. \emph{Aust. J. Stat.} $\mathbf{34}$ (1992), pp. 323--339.

\bibitem{pekoz1} Pek\"oz, E. and R\"ollin, A. New rates for exponential approximation and the theorems of R\'{e}nyi and Yaglom. \emph{Ann. Probab.} $\mathbf{39}$ (2011), pp. 587--608.

\bibitem{prr13} Pek\"oz, E., R\"ollin, A. and Ross, N.  Total variation error bounds for geometric approximation. \emph{Bernoulli} $\mathbf{19}$ (2013), pp. 610--632.


\bibitem{prr19} Pek\"oz, E., R\"ollin, A. and Ross, N. Exponential and Laplace approximation for occupation statistics of branching random walk. \emph{Electron. J. Probab.} $\mathbf{25}$ no. 55 (2020), pp. 1--22.

\bibitem{pike} Pike, J. and Ren, H. Stein's method and the Laplace distribution. \emph{ALEA Lat. Am. J. Probab. Math. Stat.} $\mathbf{11}$ (2014), pp. 571--587.

\bibitem{r98} Reinert, G. Couplings for Normal Approximations with Stein's Method. In \emph{Microsurveys in Discrete Probability}, volume of \emph{DIMACS series AMS}, (1998), pp. 193--207.

\bibitem{renyi} R\'{e}nyi, A. A characterization of Poisson processes. \emph{Magyar Tud. Akad. Mat. Kutat\'{o} Int. K\"{o}zl.} $\mathbf{1}$ (1957), pp. 519--527.


\bibitem{schoutens2} Schoutens, W. Orthogonal Polynomials in Steins Method. EURANDOM Report 99-041, EURANDOM, 1999.
  
\bibitem{schoutens} Schoutens, W.  Orthogonal polynomials in Stein's
  method.  \emph{J. Math. Anal. Appl.} $\mathbf{253}$ (2001),
  pp. 515--531. 
  
\bibitem{s10}  Shevtsova, I. An Improvement of Convergence Rate Estimates in the Lyapunov Theorem.  \emph{Dokl. Math.} $\mathbf{253}$ (2010), pp. 862--864.
  
\bibitem{s11} Shevtsova, I. On the absolute constants in the Berry Esseen type inequalities for identically distributed summands. arXiv:1111.6554, 2011.

\bibitem{stein} Stein, C.  A bound for the error in the normal approximation to the the distribution of a sum of dependent random variables.  In \emph{Proc. Sixth Berkeley Symp. Math. Statis. Prob.} (1972), vol. 2, Univ. California Press, Berkeley, pp. 583--602.

\bibitem{stein2} Stein, C.  \emph{Approximate Computation of Expectations.} IMS, Hayward, California, 1986.


\bibitem{toda} Toda, A. A. Weak limit of the geometric sum of independent but not identically distributed random variables. arXiv:1111.1786, 2011.

\end{thebibliography}
\end{document}